\documentclass[11pt]{article}

\usepackage[T1]{fontenc}
\usepackage[utf8]{inputenc}
\usepackage{lmodern}
\usepackage{microtype}
\usepackage[a4paper,margin=30mm]{geometry}
\usepackage{amsmath,amssymb,amsthm,mathtools}
\usepackage{aliascnt}
\usepackage{xcolor}
\usepackage{hyperref}
\usepackage[nameinlink,capitalise,noabbrev]{cleveref}
\usepackage{fancyhdr}

\hypersetup{
  colorlinks=true,
  linkcolor=blue!45!black,
  citecolor=green!35!black,
  urlcolor=blue!55!black,
  pdftitle={Conditional Uniformization of Kahler Surfaces},
  pdfsubject={Topology and strong Steinness of complete noncompact Kahler surfaces},
  pdfkeywords={Kahler surface, contractibility, strong Steinness, simple connectivity at infinity, radial metric, holomorphic kernel, cohomology}
}

\allowdisplaybreaks[1]
\newcommand{\statementspace}[1][18]{%
  \par\penalty-100
  \ifdim\dimexpr\pagegoal-\pagetotal\relax<#1\baselineskip
    \newpage
  \fi
}

\newtheorem{theorem}{Theorem}[section]
\newaliascnt{lemma}{theorem}
\newtheorem{lemma}[lemma]{Lemma}
\aliascntresetthe{lemma}
\newaliascnt{corollary}{theorem}
\newtheorem{corollary}[corollary]{Corollary}
\aliascntresetthe{corollary}
\newaliascnt{proposition}{theorem}
\newtheorem{proposition}[proposition]{Proposition}
\aliascntresetthe{proposition}
\theoremstyle{definition}
\newaliascnt{definition}{theorem}
\newtheorem{definition}[definition]{Definition}
\aliascntresetthe{definition}
\theoremstyle{remark}
\newaliascnt{remark}{theorem}

\aliascntresetthe{remark}

\newcommand{\R}{\mathbb R}
\newcommand{\C}{\mathbb C}
\newcommand{\Z}{\mathbb Z}
\newcommand{\BK}{\operatorname{BK}}
\newcommand{\Ric}{\operatorname{Ric}}
\newcommand{\NQOBC}{\mathrm{NQOBC}}
\newcommand{\Hom}{\operatorname{Hom}}

\newcommand{\SCI}{\mathrm{SCI}}
\newcommand{\Dcal}{\mathcal D}
\newcommand{\Hcal}{\mathcal H}
\newcommand{\Pcal}{\mathcal P}
\newcommand{\OO}{\mathcal O}
\newcommand{\ddc}{\sqrt{-1}\,\partial\bar\partial}

\title{\bfseries Conditional Uniformization of K\"ahler Surfaces}
\author{Jingcao Wu}
\date{}

\begin{document}
\maketitle

\begin{abstract}
We prove that a complete noncompact K\"ahler surface with nonnegative Ricci and nonnegative quadratic orthogonal bisectional curvature is contractible, and hence homeomorphic to $\R^4$, if it is simply connected at infinity. Under positive bisectional curvature, this removes the contractibility assumption from the conditional uniformization theorem of Datar--Pingali--Seshadri: strong Steinness and simple connectivity at infinity suffice to identify the surface biholomorphically with $\C^2$.
We also derive bounded-gradient strictly plurisubharmonic exhaustions and uniform holomorphic kernel estimates for complete $U(n)$-invariant K\"ahler metrics on $\C^n$ with nonnegative bisectional curvature.
\end{abstract}

\medskip
\noindent\textbf{2020 Mathematics Subject Classification.}
Primary 32Q15; Secondary 32E10, 32A25, 55N20.

\noindent\textbf{Keywords.}
K\"ahler surface; contractibility; simple connectivity at infinity.

\begingroup
\small
\setlength{\parskip}{0pt}
\tableofcontents
\endgroup

\section{Introduction}\label{sec:introduction}

Yau's uniformization conjecture \cite{Yau94} asks whether a complete noncompact K\"ahler manifold with positive holomorphic bisectional curvature is biholomorphic to complex Euclidean space. For surfaces, Datar--Pingali--Seshadri \cite[Theorem~2]{DPS26} prove this conclusion under three additional hypotheses: strong Steinness, contractibility,
and simple connectivity at infinity. Our main topological theorem holds under weaker curvature assumptions and makes contractibility redundant when simple connectivity at infinity is assumed.

Throughout, manifolds are connected, smooth, second countable, and without boundary. We write $(X,g)$ for a K\"ahler manifold, denote its complex structure by $J$, and reserve $\omega$ for the associated K\"ahler form, defined by $\omega(V,W)=g(JV,W)$. We write $\BK\geq0$ and $\BK>0$ for nonnegative and positive holomorphic bisectional curvature; no upper curvature bound is implicit. Simple connectivity at infinity (SCI), strong Steinness, and our metric conventions are recalled in \cref{sec:definitions}. For an abelian group $A$, set
\[
 H^q_\infty(X;A)=\varinjlim_{K\Subset X}H^q(X\setminus K;A).
\]

\begin{theorem}\label{thm:main-topology}
Let $(X,g)$ be a complete noncompact K\"ahler surface with $\Ric\geq0$ and $\NQOBC\geq0$. If $X$ is $\SCI$, then $X$ is contractible and has one end.
\end{theorem}

Here $\NQOBC\geq0$ denotes nonnegative quadratic orthogonal bisectional curvature; see \cref{def:nqobc}. This condition is strictly weaker than $\BK\geq0$, even in the presence of $\Ric\geq0$. Ni--Zheng \cite[Theorem~7.1]{NZ18} construct complete K\"ahler metrics
on $\C^2$ with positive Ricci and positive orthogonal bisectional curvature, but with negative holomorphic sectional curvature somewhere. The proof of \cref{thm:main-topology} combines the following fundamental-group theorem with a Bochner argument.

\begin{theorem}\label{thm:sci-simple}
Let $(X,g)$ be a complete noncompact K\"ahler manifold of any complex dimension with $\Ric\geq0$. If $X$ is $\SCI$, then $\pi_1(X)=0$ and $X$ has one end.
\end{theorem}

We deduce \cref{thm:sci-simple} from Sormani's loops-to-infinity and split-double-cover results \cite[Theorem~11 and Corollary~20]{Sor01}. The K\"ahler structure and Cheeger--Gromoll splitting \cite{CG71} imply that a simply connected complete noncompact K\"ahler manifold with $\Ric\geq0$ has one end, whereas a connected finite cover of an SCI manifold inherits SCI and has at least as many ends as its degree. These facts exclude the exceptional double cover. The simple-connectivity conclusion for $\BK>0$ was also noted by Chau--Tam \cite[proof of Theorem~1.1]{CT08}.

The next theorem replaces SCI by the weaker end-cohomology hypothesis needed for contractibility under $\BK\geq0$, while retaining simple connectivity.

\begin{theorem}\label{thm:refined-contractibility}
Let $(X,g)$ be a complete noncompact K\"ahler surface with $\BK\geq0$, $\pi_1(X)=0$, and $H^1_\infty(X;\R)=0$. Then $X$ is contractible.
\end{theorem}

\Cref{thm:refined-contractibility} is a topological consequence of Datar--Pingali--Seshadri's real-cohomology vanishing theorem \cite[Proposition~6.1]{DPS26}. The universal coefficient theorem, end cohomology, and integral Poincar\'e duality yield integral acyclicity; simple connectivity then gives contractibility. For \cref{thm:main-topology}, \cref{lem:weak-bochner} extends the vanishing input to $\Ric\geq0$ and $\NQOBC\geq0$. All three topological theorems are proved in \cref{sec:refined-proof}.

Freedman's recognition theorem \cite[Corollary~1.2]{Fre82} identifies the underlying topological manifold in \cref{thm:main-topology} with $\R^4$; we record this consequence in \cref{cor:euclidean-homeomorphism}. In particular, $\BK\geq0$ and SCI suffice for this conclusion.

A closely related result is due to Chau--Tam \cite[Theorem~1.1]{CT08}. They prove that a complete noncompact K\"ahler $n$-manifold with bounded positive bisectional curvature and SCI is biholomorphic to a pseudoconvex domain in $\C^n$ homeomorphic to $\R^{2n}$, provided its scalar curvature satisfies the uniform average bound
\[
\frac{1}{\operatorname{Vol}_g B_g(x,r)}\int_{B_g(x,r)}\operatorname{Scal}_g\,dV_g\leq\frac{C}{1+r}\qquad(x\in X,\ r>0).
\]
In complex dimension two, \cref{cor:euclidean-homeomorphism} gives the Euclidean homeomorphism without either supplementary curvature bound. 

Under weaker curvature assumptions, Lee--Tam \cite[Corollary~1.1]{LT21} prove that a complete noncompact K\"ahler $n$-manifold with nonnegative Ricci curvature, nonnegative orthogonal bisectional curvature, and maximal volume growth is biholomorphic to a pseudoconvex domain in $\C^n$ homeomorphic to $\R^{2n}$. In complex dimension two, their curvature assumptions coincide with those of \cref{thm:main-topology}. In \cref{cor:euclidean-homeomorphism}, SCI replaces maximal volume growth, and the conclusion is topological; Lee--Tam also obtain a holomorphic conclusion. Under $\BK\geq0$ and maximal volume growth, Liu's theorem \cite[Theorem~1.1]{Liu19} gives a biholomorphism with $\C^n$.

Topological and smooth recognition differ in real dimension four. Gompf \cite{Gom98} constructs exotic smoothings of $\R^4$ admitting Stein structures, so contractibility, SCI, and ordinary Steinness do not determine the smooth structure. Whether the curvature hypotheses exclude such smoothings remains open here. Datar--Pingali--Seshadri \cite[Introduction]{DPS26} likewise distinguish the classical diffeomorphism conclusion under positive real sectional curvature from the question under positive bisectional curvature alone.

We next derive quantitative estimates for complete rotationally symmetric metrics from the established radial curvature theory. The underlying complex manifold is $\C^n$; the estimates concern the prescribed metric and its weighted holomorphic function spaces.

\begin{theorem}\label{thm:radial}
Let $n\geq2$ and let $g$ be a complete $U(n)$-invariant K\"ahler metric on $\C^n$ with $\BK\geq0$. There is a constant $\kappa>0$ such that
\[
 g\geq\frac{\kappa}{1+|z|^2}g_{\mathrm{Eucl}}.
\]
The strictly plurisubharmonic exhaustion $\rho=\log(1+|z|^2)$ has bounded $g$-gradient. For $q=n+2$, the anchored weighted space
\[
\Hcal_q=\left\{F\in\OO(\C^n):F(0)=0,\quad\int_{\C^n}|F|^2(1+|z|^2)^{-q}\,dV_g<\infty\right\}
\]
is a finite direct sum of all homogeneous polynomial spaces of degrees $1,\ldots,d_*$, with $1\leq d_*\leq q-1$. Its diagonal kernel is
\[
 K_q(z)=K(|z|^2),\qquad K(s)=\sum_{d=1}^{d_*}c_ds^d,\quad c_d>0.
\]
The function $q^{-1}\log(1+K_q)$ is a smooth strictly plurisubharmonic exhaustion. For a suitable sequence of annuli, this single kernel satisfies annular properness \textup{(P)} and the global relative derivative estimate \textup{(J)} of \cref{sec:radial-estimates}, with $q_j=q$ and $a_j=1$ for every $j$ and a constant independent of the point and the annulus.
\end{theorem}

The metric comparison and logarithmic exhaustion in \cref{thm:radial} follow directly from the radial curvature framework of Wu--Zheng \cite{WZ11} and Yang's nonnegative-curvature characterization \cite[Proposition~3.1]{Yang13}. The calculations in \cref{sec:radial} identify the finite polynomial weighted space and verify \textup{(P)} and \textup{(J)} with a single normalization. 

Combining \cref{thm:main-topology} with the conditional uniformization theorem of Datar--Pingali--Seshadri \cite[Theorem~2]{DPS26} gives the following corollary.

\begin{corollary}\label{cor:conditional-uniformization}
Let $(X,g)$ be a complete noncompact K\"ahler surface with $\BK>0$ and $\SCI$. If $(X,g)$ is strongly Stein, then $X$ is biholomorphic to $\C^2$.
\end{corollary}

Ni--Tam's Busemann construction \cite[Lemma~4.1 and the proof of Theorem~4.2(ii)]{NT03} supplies strong Steinness under an additional sectional-curvature hypothesis.

\begin{corollary}\label{cor:sectional-infinity}
Let $(X,g)$ be a complete noncompact K\"ahler surface with $\BK>0$ and $\SCI$. Suppose that there is a compact set $K\subset X$ such that every real sectional curvature is nonnegative on $X\setminus K$. Then $(X,g)$ is strongly Stein and $X$ is biholomorphic to $\C^2$.
\end{corollary}

We prove both corollaries in \cref{sec:uniformization}. The additional hypothesis in \cref{cor:sectional-infinity} makes the Busemann function proper; heat regularization then yields a smooth plurisubharmonic exhaustion with bounded gradient.

Without SCI, the following statement collects consequences of Sormani's loops-to-infinity theorem \cite[Theorem~8]{Sor01}, Ni--Tam's strictly plurisubharmonic function \cite[Theorem~4.2(ii)]{NT03}, weighted point separation and local coordinates \cite[Proposition~2.2]{DPS26}, the Bochner--Hartogs theorem of Napier--Ramachandran \cite[Theorem~0.2]{NR18}, and V\^aj\^aitu's cohomological Stein criterion \cite[Proposition~1]{Vaj10}.

\begin{theorem}\label{thm:cohomological-reduction}
Let $(X,g)$ be a complete noncompact K\"ahler surface with $\BK>0$. Then:
\begin{enumerate}
\item $X$ has one end and the geodesic loops-to-infinity property;
\item $X$ has no compact positive-dimensional analytic subsets;
\item global holomorphic functions separate points and provide local holomorphic coordinates;
\item $H_c^1(X,\OO_X)=0$;
\item if $\dim_\C H^1(X,\OO_X)<\infty$, then $X$ is Stein.
\end{enumerate}
Consequently a non-Stein example would have $\dim_\C H^1(X,\OO_X)=\infty$.
\end{theorem}

The loops-to-infinity property means that every fundamental-group element can be represented outside arbitrarily large compact sets, with base points transported along a fixed ray. These representatives need not contract there. The proof of \cref{thm:cohomological-reduction} is given in \cref{sec:cohomological-proof}.

\section{Preliminaries}\label{sec:definitions}

\begin{definition}\label{def:strongly-stein}
A K\"ahler manifold $(X,g)$ is \emph{strongly Stein} if it admits a smooth real function $\rho:X\to\R$ such that
\[
\ddc\rho\geq0,\qquad|\nabla\rho|_g\leq1,
\]
and every sublevel $\{\rho\leq a\}$ is compact. The constant $1$ in the gradient bound may be replaced by any finite uniform bound, since rescaling $\rho$ preserves plurisubharmonicity and properness.
\end{definition}

\begin{definition}\label{def:sci}
A connected noncompact manifold $X$ is \emph{simply connected at infinity} if, for every compact set $K\subset X$, there is a compact $L\supset K$ such that every loop in $X\setminus L$ is null-homotopic in $X\setminus K$.
\end{definition}

If $X$ has one end, choose a cofinal nested compact exhaustion $K_i$ with connected complements $U_i=X\setminus K_i$. Then \cref{def:sci} is equivalent to the statement that for every $i$ there is $j>i$ such that
\[
\pi_1(U_j)\longrightarrow\pi_1(U_i)\quad\text{is the zero homomorphism},
\]
after the usual changes of base point along a proper ray. Abelianization and the universal coefficient theorem give
\[
 H^1_\infty(X;A)=0
\]
for every abelian coefficient group $A$.

For later use, compactly supported cohomology may be written as
\[
 H^q_c(X;A)=\varinjlim_iH^q(X,U_i;A).
\]
Taking the filtered direct limit of the long exact sequences of the pairs $(X,U_i)$ gives the exact sequence
\begin{equation}\label{eq:end-sequence}
\cdots\longrightarrow H^{q-1}_\infty(X;A)\longrightarrow H^q_c(X;A)\longrightarrow H^q(X;A)\longrightarrow H^q_\infty(X;A)\longrightarrow\cdots.
\end{equation}

\begin{definition}\label{def:holomorphic-properties}
A complex manifold $X$ is \emph{holomorphically separable} if any two distinct points are separated by a global holomorphic function. It is \emph{holomorphically spreadable} if, at each $x\in X$, finitely many global holomorphic functions have $x$ as an isolated point of their common level set. In particular, global holomorphic functions that provide local coordinates at each point imply holomorphic spreadability.
\end{definition}

In holomorphic coordinates, our conventions for the Riemannian metric $g$ and its associated K\"ahler form $\omega$ are
\[
g=\operatorname{Re}\!\left(\sum_{\alpha,\beta}g_{\alpha\bar\beta}\,dz_\alpha\otimes d\bar z_\beta\right)\qquad\textrm{and}\qquad\omega=\frac{i}{2}\sum_{\alpha,\beta}g_{\alpha\bar\beta}\,dz_\alpha\wedge d\bar z_\beta.
\]
We write $d_g$, $dV_g$, and $\nabla$ for the distance, volume measure, and Levi-Civita connection of $g$; on functions, $\nabla$ denotes the $g$-gradient. Unsubscripted distances and norms refer to this same metric unless another metric is specified. For a complex-valued function $F$, we set
\[
 |\nabla F|_g^2=|\nabla\operatorname{Re}F|_g^2+|\nabla\operatorname{Im}F|_g^2.
\]
For a real function $v$,
\[
|\nabla v|_g^2=4\sum_{\alpha,\beta}g^{\alpha\bar\beta}v_\alpha v_{\bar\beta}.
\]
We use $g$ in comparisons of Riemannian metrics and $\omega$ in comparisons of real $(1,1)$-forms.

The following curvature condition was formulated explicitly by Wu--Yau--Zheng \cite[Theorem~1]{WYZ09} in their study of boundary classes of K\"ahler cones.

\begin{definition}\label{def:nqobc}
A K\"ahler manifold has \emph{nonnegative quadratic orthogonal bisectional curvature}, written $\NQOBC\geq0$, if at every point, in every unitary frame, and for every real vector $(a_1,\ldots,a_n)$,
\begin{equation}\label{eq:nqobc}
\sum_{\alpha,\beta=1}^n R_{\alpha\bar\alpha\beta\bar\beta}(a_\alpha-a_\beta)^2\geq0.
\end{equation}
\end{definition}

Nonnegative holomorphic bisectional curvature implies both $\Ric\geq0$ and $\NQOBC\geq0$. In complex dimension two, \eqref{eq:nqobc} is equivalent to $R_{1\bar1 2\bar2}\geq0$ in every unitary frame, since its left-hand side is $2R_{1\bar1 2\bar2}(a_1-a_2)^2$.

\section{Simple connectivity, ends, and contractibility}
\label{sec:refined-proof}

\subsection{Simple connectivity and ends}
\label{sec:sci-and-ends}

We first prove \cref{thm:sci-simple} using the following two propositions.
\begin{proposition}\label{prop:one-end}
A complete noncompact simply connected K\"ahler manifold with $\Ric\geq0$ has exactly one end.
\end{proposition}

\begin{proof}
Suppose $Y$ has at least two ends. Taking a limit of minimizing geodesics joining points escaping through distinct ends produces a line. The Cheeger--Gromoll splitting theorem \cite{CG71} gives an isometry $Y=N\times\R$. The factor $N$ must be compact: if $N$ were noncompact, the complement of a compact product $C\times[-a,a]$ would be connected through the regions $N\times(a,\infty)$, $N\times(-\infty,-a)$, and $(N\setminus C)\times\R$, giving one end. Such compact products are cofinal among compact sets.

Let $V=\partial_t$ be the parallel unit vector in the $\R$ direction. Since $\nabla J=0$, the vector $JV$ is also parallel. It is orthogonal to $V$, so its restriction to $N\times\{0\}$ is tangent to $N$. The one-form $\alpha=g(JV,\cdot)|_{TN}$ is parallel, closed, and has constant nonzero norm. Since $N\times\R$ is simply connected, so is $N$. Thus $\alpha=du$ for a smooth function $u$ on $N$. At a maximum of $u$ on compact $N$, $du=0$, a contradiction.
\end{proof}

\begin{proposition}\label{prop:sci-covers}
Let $p:\widehat X\to X$ be a connected finite cover of degree $d$, where $X$ is a connected noncompact manifold satisfying $\SCI$. Then $\widehat X$ is $\SCI$ and has at least $d$ ends.
\end{proposition}

\begin{proof}
Given a compact set $C\subset\widehat X$, apply SCI in $X$ to $p(C)$ and choose a compact set $L\supset p(C)$ such that loops outside $L$ contract outside $p(C)$. A loop outside $p^{-1}(L)$ projects to such a loop. The null-homotopy lifts to a disk outside $p^{-1}(p(C))$, hence outside $C$. The set $p^{-1}(L)$ is compact because the cover is finite. Thus $\widehat X$ is SCI.

SCI also gives a compact set $L$ such that every loop outside $L$ is null-homotopic in $X$. Enlarge $L$ to a compact smooth domain, so its complement has finitely many components, and let $U$ be a component that is not relatively compact. Since $\pi_1(U)\to\pi_1(X)$ is trivial, covering-space monodromy gives
\[
p^{-1}(U)=U_1\sqcup\cdots\sqcup U_d,\qquad p|_{U_\nu}:U_\nu\longrightarrow U\text{ a homeomorphism}.
\]
These are distinct components of $\widehat X\setminus p^{-1}(L)$, and none is relatively compact: a relatively compact lift would have relatively compact image. Each such component contains an end, and distinct components determine distinct ends. Hence $\widehat X$ has at least $d$ ends.
\end{proof}

\begin{proof}[Proof of \cref{thm:sci-simple}]
Suppose $\pi_1(X)\neq0$. Fix a nontrivial element and a ray from its base point. SCI implies that all sufficiently remote loops are null-homotopic in $X$. The chosen nontrivial element therefore fails the loops-to-infinity property along the ray: change of base point preserves the trivial class, so a loop that is null-homotopic in $X$ cannot represent the chosen element.

Sormani's double-cover theorem \cite[Theorem~11]{Sor01} produces a connected double cover with an isometric splitting
\[
 \widehat X=N\times\R.
\]
By \cref{prop:sci-covers}, $\widehat X$ is SCI. Every loop in the product can be translated in the $\R$ direction beyond any given compact set. Since SCI makes sufficiently remote loops null-homotopic in $\widehat X$, translation shows that every loop in $\widehat X$ is null-homotopic. Thus $\widehat X$ is simply connected. The pulled-back metric and complex structure are complete K\"ahler with $\Ric\geq0$, so \cref{prop:one-end} gives one end. But \cref{prop:sci-covers} gives at least two ends, a contradiction. Consequently $\pi_1(X)=0$. Applying \cref{prop:one-end} directly to $X$
now proves one-endedness.
\end{proof}

\subsection{Integral acyclicity and contractibility}
\label{sec:integral-contractibility}

We now turn from simple connectivity to contractibility and prove \cref{thm:refined-contractibility}. Datar--Pingali--Seshadri's real-cohomology vanishing theorem \cite[Proposition~6.1]{DPS26}, combined with integral Poincar\'e duality, yields integral acyclicity. Simple connectivity then gives contractibility.

\begin{proof}[Proof of \cref{thm:refined-contractibility}]
Tracing the bisectional curvature in a unitary frame gives
\[
 \Ric(Z,\bar Z)=\sum_{a=1}^2R(Z,\bar Z,e_a,\bar e_a)\geq0.
\]
Since $X$ is simply connected, \cref{prop:one-end} gives one-endedness. Fix a cofinal compact exhaustion $K_i$ with connected complements $U_i=X\setminus K_i$.

The coefficient inclusion $\Z\hookrightarrow\R$ induces an injection
\[
 H^1(U_i;\Z)\longrightarrow H^1(U_i;\R).
\]
Indeed, since $U_i$ is connected, the universal coefficient theorem identifies these groups naturally with $\Hom(H_1(U_i;\Z),\Z)$ and $\Hom(H_1(U_i;\Z),\R)$, respectively.
Filtered direct limits are exact, so there is an injection
\[
 H^1_\infty(X;\Z)\longrightarrow H^1_\infty(X;\R).
\]
The real end-cohomology hypothesis therefore implies $H^1_\infty(X;\Z)=0$.

By \cite[Proposition~6.1]{DPS26}, a complete noncompact K\"ahler $n$-manifold with $\BK\geq0$ and $H^1_\infty(X;\R)=0$ satisfies $H^{2n-2}(X;\R)=0$. Here $n=2$, so
\begin{equation}\label{eq:H2-real-zero}
 H^2(X;\R)=0.
\end{equation}

Set $G=H_2(X;\Z)$. Since $\pi_1(X)=0$, the degree-one Hurewicz theorem gives
\[
H_1(X;\Z)=0.
\]
The degree-two universal coefficient theorem then gives natural identifications
\[
H^2(X;\Z)\cong\Hom(G,\Z),\qquad H^2(X;\R)\cong\Hom(G,\R).
\]
The inclusion $\Z\hookrightarrow\R$ induces an injection from the first group into the second. Thus \eqref{eq:H2-real-zero} implies
\[
H^2(X;\Z)=0.
\]
The integral end sequence \eqref{eq:end-sequence} now contains
\[
H^1_\infty(X;\Z)\longrightarrow H^2_c(X;\Z)\longrightarrow H^2(X;\Z).
\]
Both outer groups vanish, so $H^2_c(X;\Z)=0$. Integral Poincar\'e duality on the oriented real four-manifold $X$ yields
\[
 H_2(X;\Z)\cong H^2_c(X;\Z)=0.
\]

It remains to treat degrees three and four. Because $X$ and all $U_i$ are connected, the restriction
\[
 H^0(X;\Z)\longrightarrow H^0_\infty(X;\Z)
\]
is an isomorphism. Also $H^1(X;\Z)=0$. The degree-zero part of \eqref{eq:end-sequence} gives $H^1_c(X;\Z)=0$, and duality gives $H_3(X;\Z)=0$. A compactly supported locally constant function on a connected noncompact manifold is zero; hence $H^0_c(X;\Z)=0$ and
$H_4(X;\Z)=0$. A smooth triangulation supplies a CW model of dimension at most four, so all higher homology groups vanish as well. Thus
\[
\widetilde H_k(X;\Z)=0\qquad(k\geq0).
\]

If some homotopy group were nonzero, choose the least $k\geq2$ with $\pi_k(X)\neq0$. The Hurewicz theorem would give $\pi_k(X)\cong H_k(X;\Z)=0$, a contradiction. All positive-degree homotopy groups vanish, and Whitehead's theorem applied to a CW model of $X$ shows that $X$ is contractible; see \cite{Hat02} for the duality, universal coefficient, Hurewicz, and Whitehead theorems used here.
\end{proof}

\subsection{Bochner vanishing and contractibility}
\label{sec:weak-topology}

We now prove \cref{thm:main-topology} under $\Ric\geq0$ and $\NQOBC\geq0$. The preceding integral-cohomology argument applies once we establish vanishing of the second real cohomology group. The role of $\NQOBC$ in the Bochner formula for $(1,1)$-forms is classical; see \cite[Lemma~1.1]{CT12} for the compact case. The following lemma uses $L^2$-cutoffs and the same curvature terms as \cite[Proposition~6.1]{DPS26} to obtain the required vanishing under the weaker curvature assumptions.

\begin{lemma}\label{lem:weak-bochner}
Let $(X,g)$ be a complete noncompact K\"ahler manifold of complex dimension $n$, with $\Ric\geq0$ and $\NQOBC\geq0$. Every real harmonic two-form in $L^2$ vanishes. If, in addition, $H^1_\infty(X;\R)=0$, then
\[
 H^{2n-2}(X;\R)=0.
\]
\end{lemma}

\begin{proof}
Let $\beta$ be a real $L^2$-harmonic two-form. The K\"ahler identities preserve harmonicity under the type decomposition
\[
 \beta=\beta^{2,0}+\beta^{1,1}+\beta^{0,2}.
\]
For the real $(1,1)$-part, choose a unitary coframe at a point and write
\[
\beta^{1,1}=i\sum_{j=1}^n\lambda_j\theta^j\wedge\bar\theta^j,\qquad \lambda_j\in\R.
\]
The curvature contribution to its Bochner formula is a positive normalization constant times
\[
 \sum_{i,j}R_{i\bar i j\bar j}(\lambda_i-\lambda_j)^2,
\]
which is nonnegative by \eqref{eq:nqobc}.

For the $(2,0)$-part, the K\"ahler identities and $L^2$-integration by parts on the complete manifold give $\bar\partial\beta^{2,0}=0$. Thus $\beta^{2,0}$ is a holomorphic two-form. In a unitary coframe diagonalizing the Ricci tensor, with eigenvalues $r_1,\ldots,r_n\geq0$, write $\beta^{2,0}=\sum_{i<j}b_{ij}\theta^i\wedge\theta^j$. Its Bochner curvature term is a positive normalization constant times
\[
 \sum_{i<j}(r_i+r_j)|b_{ij}|^2\geq0.
\]
The $(0,2)$-part is the complex conjugate of $\beta^{2,0}$. Consequently each type component $\sigma$ satisfies
\[
 \frac12\Delta|\sigma|^2\geq|\nabla\sigma|^2.
\]
The same curvature terms appear in \cite[Proposition~6.1]{DPS26}.

For completeness, take compactly supported Lipschitz cutoffs $\chi_R$ with $\chi_R=1$ on $B(o,R)$ and $|\nabla\chi_R|\leq C/R$. Integration by parts and Cauchy--Schwarz give
\[
\int_X\chi_R^2|\nabla\sigma|^2\,dV_g\leq4\int_X|\nabla\chi_R|^2|\sigma|^2\,dV_g\leq\frac{C}{R^2}\|\sigma\|_{L^2}^2.
\]
Letting $R\to\infty$ proves that $\sigma$ is parallel. A complete noncompact manifold with $\Ric\geq0$ has infinite volume, so a parallel $L^2$ form is zero. Hence $\beta=0$.

We now prove the cohomological conclusion. Let $\alpha$ be a smooth real closed two-form with compact support. The $L^2$-Hodge decomposition on a complete manifold gives
\[
 \alpha=\beta+\lim_{k\to\infty}d\gamma_k\quad\text{in }L^2,
\]
with $\beta$ a real $L^2$-harmonic two-form and $\gamma_k$ smooth one-forms. The first part gives $\beta=0$. To show that $\alpha$ is exact in ordinary de Rham cohomology, let $\zeta$ be any smooth compactly supported closed $(2n-2)$-form. Then
\[
\int_X\alpha\wedge\zeta=\lim_{k\to\infty}\int_Xd\gamma_k\wedge\zeta=0.
\]
The limit follows from $L^2$-convergence and compact support, and each integral on the right vanishes by Stokes' theorem. Nondegeneracy of the de Rham pairing between $H^2(X;\R)$ and $H_c^{2n-2}(X;\R)$ gives $[\alpha]=0$ in ordinary cohomology. Thus $\alpha=d\gamma$ for a smooth one-form $\gamma$.

Outside the support of $\alpha$, the form $\gamma$ is closed. The hypothesis $H^1_\infty(X;\R)=0$ makes it exact outside a larger compact set $L$: there $\gamma=df$. Choose a smooth cutoff $\chi$ that vanishes on a neighborhood of $L$ and equals one outside a still larger compact set. Extending $\chi f$ by zero gives a smooth global function. The form
\[
 \gamma_c=\gamma-d(\chi f)
\]
has compact support and satisfies $d\gamma_c=\alpha$. Therefore $H_c^2(X;\R)=0$. Poincar\'e duality gives $H_{2n-2}(X;\R)=0$, and the universal coefficient theorem over $\R$ gives the asserted vanishing of $H^{2n-2}(X;\R)$.
\end{proof}

\begin{proof}[Proof of \cref{thm:main-topology}]
By \cref{thm:sci-simple}, SCI and $\Ric\geq0$ imply $\pi_1(X)=0$ and one-endedness. SCI also gives $H^1_\infty(X;\R)=H^1_\infty(X;\Z)=0$. \Cref{lem:weak-bochner} supplies $H^2(X;\R)=0$. The universal coefficient theorem, the integral end sequence, and Poincar\'e duality give successively
\[
 H^2(X;\Z)=0,\qquad H_c^2(X;\Z)=0,\qquad H_2(X;\Z)=0,
\]
exactly as in the proof of \cref{thm:refined-contractibility}. Simple connectivity gives $H_1(X;\Z)=0$; one-endedness and noncompactness give $H_3(X;\Z)=H_4(X;\Z)=0$. Higher homology vanishes by the four-dimensional CW model. Thus $X$ is integrally acyclic, and the Hurewicz--Whitehead argument used there proves contractibility.
\end{proof}

Freedman's recognition theorem now gives the following topological consequence.

\begin{corollary}\label{cor:euclidean-homeomorphism}
Let $(X,g)$ be a complete noncompact K\"ahler surface with $\Ric\geq0$ and $\NQOBC\geq0$. If $X$ is simply connected at infinity, then $X$ is homeomorphic to $\R^4$.
\end{corollary}

\begin{proof}
By \cref{thm:main-topology}, $X$ is contractible. In particular, $\pi_1(X)=0$ and $H_2(X;\Z)=0$. Together with SCI, these satisfy the recognition hypotheses of Freedman \cite[Corollary~1.2]{Fre82}, which imply that the open real four-manifold $X$ is homeomorphic to $\R^4$.
\end{proof}

\section{The complete \texorpdfstring{$U(n)$}{U(n)}-invariant class}\label{sec:radial}

We derive \cref{thm:radial} from the Wu--Zheng radial formulas \cite{WZ11} and Yang's nonnegative-curvature characterization \cite[Proposition~3.1]{Yang13}. Write $s=|z|^2$ and use a radial potential
\[
\omega=\frac{i}{2}\partial\bar\partial P(s),\qquad f=P'(s),\qquad h=(sf)'=f+sf'.
\]
The positive functions $f,h$ are smooth on $[0,\infty)$, with $f(0)=h(0)>0$, and
\[
g_{\alpha\bar\beta}=f(s)\delta_{\alpha\beta}+f'(s)\bar z_\alpha z_\beta.
\]
At $z\ne0$ the eigenvalue in the complex radial direction is $h(s)$, and the other $n-1$ eigenvalues are $f(s)$.

\subsection{Curvature and the metric lower bound}

Set $\xi(s)=-sh'(s)/h(s)$. The radial characterization gives
\begin{equation}\label{eq:xi-conditions}
 \xi(0)=0,\qquad \xi'\geq0\qquad\textrm{and}\qquad\xi\leq1.
\end{equation}
See \cite[Section~2 and Proposition~3.1]{Yang13} for the extension of the Wu--Zheng conditions to nonnegative bisectional curvature. For completeness, we give a direct proof of the inequalities used below. The radial holomorphic curvature component is $A=\xi'/h$, so $\BK\geq0$ implies $\xi'\geq0$ and hence $\xi\geq0$. If $\xi(s_0)>1$, monotonicity gives $\xi\geq1+\delta$ for $s\geq s_0$, for some $\delta>0$. Integrating $h'/h=-\xi/s$ gives $h(s)\leq Cs^{-1-\delta}$. Then the radial length to infinity $\frac12\int_{s_0}^\infty\sqrt{h(s)/s}\,ds$ is finite, contrary to completeness. Thus \eqref{eq:xi-conditions} holds without a bounded curvature hypothesis.

It follows that $h'\leq0$ and $(sh)'=h(1-\xi)\geq0$. Let $\kappa=h(1)>0$. For $s\geq1$, $sh(s)\geq\kappa$; for $s\leq1$, $h(s)\geq\kappa$. Since
\[
 f(s)=\frac1s\int_0^s h(t)\,dt\geq h(s),
\]
with the limiting value at $s=0$, we obtain
\[
g\geq\frac{\kappa}{1+s}g_{\mathrm{Eucl}}\qquad\textrm{and}\qquad h(s)\leq f(s)\leq h(0).
\]
This compares the prescribed metric with the Euclidean metric in the given holomorphic coordinates.

The function $\rho=\log(1+s)$ is smooth, proper, and strictly plurisubharmonic, since its Euclidean complex Hessian has eigenvalues $(1+s)^{-2}$ and $(1+s)^{-1}$. Moreover,
\[
 |\nabla\rho|_g^2=\frac{4s}{(1+s)^2h(s)}\leq\frac4\kappa.
\]
Thus $(\C^n,g)$ is strongly Stein.

\subsection{The weighted space and its polynomial kernel}

The metric volume satisfies
\begin{equation}\label{eq:volume-density}
dV_g=f^{n-1}h\,dV_{\mathrm{Eucl}}\qquad\textrm{and}\qquad\frac{\kappa^n}{(1+s)^n}\leq f^{n-1}h\leq h(0)^n.
\end{equation}
For a multi-index $\alpha$ of total degree $d$, the squared weighted norm of $z^\alpha$ is a positive angular constant times
\[
 M_d=\int_0^\infty s^{d+n-1}(1+s)^{-q}f(s)^{n-1}h(s)\,ds.
\]
Take $q=n+2$. By the upper bound in \eqref{eq:volume-density}, all linear monomials are integrable: the integrand at infinity is at most a constant times $s^n(1+s)^{-n-2}=O(s^{-2})$. By the lower bound, for $s\geq1$ it is at least a positive constant times $s^{d-q-1}$. Thus $M_d=\infty$ when $d\geq q$.

The radial measure is invariant under coordinatewise rotations, so distinct monomials are orthogonal. Applying Parseval's identity on the angular variables and Tonelli's theorem to the Taylor series of an entire function expresses its squared norm as the sum of the squared coefficients times the squared monomial norms. An entire function in the weighted $L^2$ space therefore has no Taylor terms of nonintegrable degree. Integrability at a degree implies integrability at each lower degree, since the origin causes no divergence and $s^{e}\leq s^d$ for $s\geq1$ and $e\leq d$. Hence for some $1\leq d_*\leq q-1$,
\[
\Hcal_q=\bigoplus_{d=1}^{d_*}\Pcal_d\qquad\textrm{and}\qquad\dim\Hcal_q=\binom{n+d_*}{n}-1,
\]
where $\Pcal_d$ is the full space of homogeneous holomorphic polynomials of degree $d$. The condition $F(0)=0$ excludes the constant term.

For an orthonormal basis $s_1,\ldots,s_N$ of $\Hcal_q$, define
\[
K_q(z)=\sum_{\nu=1}^N|s_\nu(z)|^2,\qquad\Dcal_q(z)=\sum_{\nu=1}^N|\nabla s_\nu(z)|_g^2.
\]
Both quantities are independent of the choice of orthonormal basis. The diagonal kernel on each degree is $U(n)$-invariant and homogeneous of bidegree $(d,d)$, and hence is $c_d|z|^{2d}$ for some $c_d>0$. Consequently
\[
 K_q(z)=K(s)=\sum_{d=1}^{d_*}c_ds^d.
\]
In particular $K(0)=0$, $K'(s)>0$, and $K(s)\to\infty$. Let $S=(s_1,\ldots,s_N)$. Since $\Hcal_q$ contains all linear polynomials, $dS$ is injective at every point. For a $(1,0)$-vector $v$, direct differentiation and Cauchy--Schwarz give
\[
(\partial\bar\partial\log(1+|S|^2))(v,\bar v)=\frac{(1+|S|^2)|dS(v)|^2-|\langle dS(v),S\rangle|^2}{(1+|S|^2)^2}\geq\frac{|dS(v)|^2}{(1+|S|^2)^2}.
\]
Consequently $u=q^{-1}\log(1+K_q)$ is strictly plurisubharmonic. It is also smooth and proper.

\subsection{The global derivative estimate and annular properness}
\label{sec:radial-estimates}

We specify the notation for the two estimates in \cref{thm:radial}. Given radii $0<R_1<R_2<\cdots\to\infty$, set $A_j=\{z:R_j\leq d_g(0,z)\leq R_{j+1}\}$. For kernels $K_j$, their derivative sums $\Dcal_j$, exponents $q_j\geq1$, and constants $a_j>0$, the estimates are
\begin{align}
\frac1{q_j}\log(1+a_jK_j(z))&\geq j&&(z\in A_j),\tag{P}\label{eq:P}\\
a_j\Dcal_j(z)&\leq C_0q_j^2(1+a_jK_j(z))&&(z\in\C^n),\tag{J}\label{eq:J}
\end{align}
where $C_0<\infty$ is independent of $j$ and $z$. We will take the same space and kernel at every stage: $q_j=q=n+2$, $K_j=K_q$, $\Dcal_j=\Dcal_q$, and $a_j=1$.

For holomorphic basis functions, $\partial_\alpha\partial_{\bar\beta}K_q=\sum_\nu(\partial_\alpha s_\nu)\overline{\partial_\beta s_\nu}$. The complex Hessian of $K(s)$ has eigenvalues $K'+sK''$ and $K'$. With our complex-valued gradient convention this gives
\[
\Dcal_q(s)=2\left\{\frac{K'(s)+sK''(s)}{h(s)}+(n-1)\frac{K'(s)}{f(s)}\right\}.
\]
For $s\geq1$, $f,h\geq\kappa/s$, and hence
\begin{align}
\Dcal_q(s)\leq\frac2\kappa\{s^2K''(s)+nsK'(s)\}\notag=\frac2\kappa\sum_{d=1}^{d_*}d(d+n-1)c_ds^d\notag\leq\frac{2d_*(d_*+n-1)}\kappa K(s).
\end{align}
For $0\leq s\leq1$, $f,h\geq\kappa$, so
\[
\Dcal_q(s)\leq\frac2\kappa\sum_{d=1}^{d_*}d(d+n-1)c_d=:B_g<\infty.
\]
It follows, for example with $C_g=q^{-2}\max\{B_g,2d_*(d_*+n-1)/\kappa\}$, that
\[
 \Dcal_q(z)\leq C_gq^2(1+K_q(z))\qquad(z\in\C^n).
\]
This proves \textup{(J)} on all of $\C^n$, including the origin, with a constant depending only on the metric and uniform over all annuli. Moreover, $|\nabla K_q|_g\leq2\sqrt{K_q\Dcal_q}$ gives
\[
|\nabla u|_g\leq\frac{2\sqrt{K_q\Dcal_q}}{q(1+K_q)}\leq2\sqrt{C_g}.
\]
Thus the logarithmic kernel exhaustion also has bounded gradient.

The radial distance from the origin is
\[
 R(s)=d_g(0,z)=\frac12\int_0^s\sqrt{\frac{h(t)}t}\,dt.
\]
Indeed, the length of any path is bounded below by the integral of its radial speed, and a radial segment attains this bound. Completeness implies $R(s)\to\infty$. The function
\[
u(s)=q^{-1}\log(1+K(s))
\]
is strictly increasing from zero to infinity. Define $s_j$ by $u(s_j)=j$ and set $R_j=R(s_j)$ for $j\geq1$. On $A_j=\{R_j\leq R\leq R_{j+1}\}$,
\[
 \frac1q\log(1+K_q(z))\geq j.
\]
Thus \textup{(P)} and \textup{(J)} hold with the same choices $K_j=K_q$, $q_j=n+2$, and $a_j=1$ at every stage. This completes the proof of \cref{thm:radial}.

\section{Consequences for uniformization}
\label{sec:uniformization}

We now prove the two uniformization corollaries. The first combines our topological theorem with the conditional result of Datar--Pingali--Seshadri. The second uses Ni--Tam's Busemann construction to obtain strong Steinness under an additional sectional-curvature hypothesis.

\begin{proof}[Proof of \cref{cor:conditional-uniformization}]
Nonnegative bisectional curvature implies $\Ric\geq0$ and $\NQOBC\geq0$, so \cref{thm:main-topology} makes $X$ contractible. By hypothesis it is also strongly Stein and simply connected at infinity. The conditional uniformization theorem \cite[Theorem~2]{DPS26} therefore gives a biholomorphism $X\cong\C^2$.
\end{proof}

For a point $o\in X$ and a unit-speed minimizing ray $\gamma$ from $o$, write
\[
b_\gamma(x)=\lim_{t\to\infty}\bigl(t-d_g(x,\gamma(t))\bigr),\qquad B(x)=\sup_{\gamma(0)=o}b_\gamma(x).
\]
The function $B$ is $1$-Lipschitz. Under nonnegative holomorphic bisectional curvature it is plurisubharmonic by Wu's construction \cite{Wu79}, as used in \cite[Section~4]{NT03}.
Let $H(x,y,\tau)$ denote the heat kernel of $g$, and write
\[
 P_\tau F(x)=\int_X H(x,y,\tau)F(y)\,dV_g(y).
\]

\begin{proof}[Proof of \cref{cor:sectional-infinity}]
As recalled in the proof of \cite[Theorem~4.2]{NT03}, nonnegative sectional curvature outside a compact set makes $B$ an exhaustion. The positive bisectional curvature hypothesis and \cite[Lemma~4.1 and the proof of Theorem~4.2(ii)]{NT03} give, for sufficiently small $\tau>0$,
\[
 u=P_\tau B,\qquad \ddc u>0,\qquad |\nabla u|_g\leq1,
\]
with $u$ smooth. Since $\Ric\geq0$, the heat kernel has total mass one. The $1$-Lipschitz bound and the heat-kernel first-moment estimate \cite[Lemma~2.12]{DPS26} imply
\[
|u(x)-B(x)|\leq\int_X H(x,y,\tau)d_g(x,y)\,dV_g(y)\leq C_n\sqrt\tau.
\]
Thus $u$ remains an exhaustion, and $(X,g)$ is strongly Stein. \Cref{cor:conditional-uniformization} now yields $X\cong\C^2$ biholomorphically.
\end{proof}

The sectional-curvature hypothesis enters through properness of the Busemann function. The uniform heat-kernel error bound then preserves properness under regularization. Obtaining strong Steinness from $\BK>0$ and SCI alone remains unresolved.

\section{The case without SCI and the cohomological reduction}
\label{sec:cohomological-proof}

We conclude by proving \cref{thm:cohomological-reduction} from the results cited in the introduction. The argument describes what $\BK>0$ implies without SCI and clarifies the distinction between ordinary Steinness and the bounded-gradient requirement.

\begin{proof}[Proof of \cref{thm:cohomological-reduction}]
Tracing $\BK>0$ gives $\Ric>0$. If $X$ had at least two ends, the standard construction of a line and the Cheeger--Gromoll splitting theorem \cite{CG71} would give a parallel line direction on which the Ricci tensor vanishes, a contradiction. Thus $X$ has one end.
\cite[Theorem~8]{Sor01} gives the geodesic loops-to-infinity property. This proves part~(1).

By \cite[Theorem~4.2(ii)]{NT03}, $X$ carries a smooth strictly plurisubharmonic function $\psi_0$ with bounded gradient. If there were a compact irreducible analytic curve, restriction to its compact normalization would make $\psi_0$ subharmonic and therefore constant. At a regular point of the curve its Levi form is strictly positive, a contradiction. Every compact positive-dimensional analytic subset of a noncompact complex surface contains such a curve. This proves part~(2).

The smallest Levi eigenvalue of $\psi_0$ has a positive continuous minorant. For each prescribed pair of distinct points, apply \cite[Proposition~2.2(2)]{DPS26} with a sufficiently large multiple $q\psi_0$, where $q$ may depend on that pair, to obtain a global holomorphic function separating the two points. Similarly, for each $x\in X$, \cite[Proposition~2.2(3)]{DPS26}, with $q$ chosen for $x$, gives two global holomorphic functions whose differentials are linearly independent at $x$. They form local coordinates near $x$ by the holomorphic inverse function theorem. This proves part~(3).

The one-ended complete K\"ahler manifold $X$ admits a smooth plurisubharmonic function that is strictly plurisubharmonic on every complex two-dimensional germ. The Bochner--Hartogs theorem of Napier--Ramachandran \cite[Theorem~0.2]{NR18} therefore gives
\[
 H_c^1(X,\OO_X)=0.
\]
This is part~(4).

By part~(3) and \cref{def:holomorphic-properties}, $X$ is holomorphically spreadable. Finally, V\^aj\^aitu's cohomological Stein criterion \cite[Proposition~1]{Vaj10} states that a holomorphically spreadable complex space of pure dimension $n$ is Stein when
\[
\dim_\C H^q(X,\OO_X)<\infty,\qquad 1\leq q\leq n-1.
\]
For the surface $X$, only $q=1$ occurs. This proves part~(5).
\end{proof}

\section*{Acknowledgments}

The author thanks Prof. Jixiang Fu for helpful discussions and valuable suggestions. ChatGPT-6 assisted in developing the proofs. The author has carefully verified all arguments.

\noindent\begin{minipage}{\textwidth}
\bigskip

\noindent\begin{minipage}{\textwidth}
\textsc{Jingcao Wu}\\
School of Mathematics\\
Shanghai University of Finance and Economics\\
Shanghai 200433, People's Republic of China\\
Email:
\href{mailto:wujincao@shufe.edu.cn}
{\texttt{wujincao@shufe.edu.cn}}
\end{minipage}

\end{minipage}

\end{document}